\documentclass[a4paper, 12pt]{amsart}

\usepackage[T1]{fontenc}
\usepackage[ngerman,english]{babel}
\usepackage{amsmath}
\usepackage{amsthm}
\usepackage{amssymb}
\usepackage{amsfonts}
\usepackage{bbm}
\usepackage[breaklinks]{hyperref}
\usepackage{tabularx}
\usepackage{multirow}

\usepackage{tikz}
\usetikzlibrary{matrix,arrows,patterns,shapes}
\usepackage{MnSymbol} 
\usepackage{nicefrac}
\usepackage{paralist} 

\theoremstyle{plain}
\newtheorem{theorem}{\bf{\textsc{Theorem}}}[section]
\newtheorem*{theorem*}{\bf{\textsc{Theorem}}}

\newtheorem{corollary}[theorem]{\bf{\textsc{Corollary}}}
\newtheorem{lemma}[theorem]{\bf{\textsc{Lemma}}}
\newtheorem{proposition}[theorem]{\bf{\textsc{Proposition}}}

\numberwithin{equation}{section}

\theoremstyle{definition}

\newtheorem{example}[theorem]{\bf{\textsc{Example}}}

\theoremstyle{definition}
\newtheorem{remark}[theorem]{\bf{\textsc{Remark}}}

\newcommand{\kk}{\mathbbm{k}}				
\newcommand{\CC}{\mathbb{C}}				
\newcommand{\RR}{\mathbb{R}}				

\newcommand{\Poly}{\mathcal P}			

\DeclareMathOperator{\End}{End}			
\DeclareMathOperator{\Hom}{Hom}			
\DeclareMathOperator{\GL}{GL}				
\DeclareMathOperator{\U}{U}					

\DeclareMathOperator{\Id}{Id}				

\renewcommand{\Re}{\operatorname{Re}}			
\DeclareMathOperator{\Tr}{Tr}				
\DeclareMathOperator{\Det}{Det}			

\renewcommand{\b}[1]{{\bf #1}}

\renewcommand{\Bar}{\overline}											
\newcommand{\JTP}[3]{\left\{#1,\,#2,\,#3\right\}}		
\newcommand{\B}[2]{B_{#1,\,#2}}											
\newcommand{\Set}[2]{\left\{#1\,\middle|\,#2\right\}} 
\newcommand{\sMatrix}[4]{\left(\begin{smallmatrix}#1&#2\\#3&#4\end{smallmatrix}\right)}

\tikzset{inner sep=0pt, 
  root/.style={circle,draw,minimum size=5pt,thick}, 
  cross/.style={cross out,draw,minimum size=4pt,thick},
  doubleline/.style={double distance=1.5pt,thick},
}

\title
{Inequalities for generalized minors}
\author{Benjamin Schwarz} 


\address{Benjamin Schwarz, Universit\"{a}t Paderborn,
Fakult\"{a}t f\"{u}r Elektrotechnik, Informatik und Mathematik,
Institut f\"{u}r Mathematik, Warburger Str. 100,
33098 Paderborn, Germany}
\email{bschwarz@math.upb.de}

\begin{document}
\begin{abstract}
It is a classical result that the absolute value of any $k$-minor of an $r\times s$ real or complex matrix is bounded by the product of its first $k$ singular values. We generalize this statement to the context of real or complex simple Jordan pairs with generalized minors given by Jordan algebra determinants.
\end{abstract}
\maketitle

\section*{Introduction}
The goal of this paper is a generalization of the following statement.
\begin{quote}\itshape
	Let $A$ be a real or complex $r\times s$ matrix, $r\leq s$, with singular values 
	$\sigma_1\geq\cdots\geq\sigma_r\geq0$. Then for $1\leq k\leq r$ the absolute value of any 
	$k$-minor of $A$ is bounded by the product $\sigma_1\cdots\sigma_k$.	
\end{quote}

This estimate can be proved using the Cauchy--Binet formula, see e.g.\ \cite[Theorem~4.1]{Br56}. We generalize this statement to the context of Jordan theory. Let $(V,\Bar V)$ be a real or complex simple Jordan pair with positive involution. Then any element $z\in V$ admits a singular value decomposition, and to any tripotent $e\in V$ there is associated a unital Jordan algebra $[e]\subseteq V$ and a Jordan algebra determinant $\Delta_e$.

\begin{theorem*}
	Let $(V,\Bar V)$ be a real or complex simple Jordan pair with positive involution and 
	rank $r$. Let $e\in V$ be a tripotent of rank $k$, and $z\in V$ be an element with singular 
	values $\sigma_1\geq\cdots\geq\sigma_r\geq 0$. Then
	\[
		\left|\Delta_e(z)\right| \leq \sigma_1\cdots\sigma_k.
	\]
\end{theorem*}
Our proof is analytic and quite elementary. Let us describe the main idea in the classical context with $V=\kk^{r\times s}$, $\kk=\RR$ or $\CC$. Then an element $e\in\kk^{r\times s}$ is tripotent if it satisfies the identity $e = ee^*e$, and it turns out that the corresponding Jordan algebra determinant $\Delta_e$ is given by
\[
	\Delta_e(z) = \Det(\b 1_r + (e-z)e^*)\quad(z\in\kk^{r\times s}).
\]
Moreover, the set $S_k$ of all rank-$k$ tripotents forms a real smooth compact submanifold of $\kk^{r\times s}$. Now the theorem follows from plain analysis of the smooth map $f:S_k\to\RR$ given by $f(e) = |\Delta_e(z)|^2$ with fixed $z\in\kk^{r\times s}$.

We note that choosing the tripotent $e$ appropriately, $\Delta_e(z)$ coincides with a given $k$-minor of $z$, see Example~\ref{ex:usualminors} for details. Therefore, the theorem indeed generalizes the classical statement.

In the first section, we discuss generalized minors in the context of simple Jordan pairs over arbitrary fields of characteristic $\neq 2$ in some detail. We note that even though the basic definition of generalized minors (resp.\ the corresponding Jordan algebra determinants) is widely known, our results seem to be new. In Section~\ref{sec:inequalities} we specialize to the case of real or complex simple Jordan pairs, prove our main theorem, and provide an application involving some representation theory on the space of polynomials on $V$, see Corollary~\ref{cor:growthestimate}.

\section{Generalized minors}
In this section we consider Jordan pairs without fixing an involution. Let $(V^+,V^-)$ be a finite dimensional simple Jordan pair over a field $\kk$ of characteristic $\neq 2$, and with quadratic maps $Q^\pm:V^\pm\to\Hom(V^\mp,V^\pm)$. As usual, we omit the signs and simply write $Q_z:=Q^\pm(z)$ for $z\in V^\pm$. Moreover, for $x,z\in V^\pm$ and $y\in V^\mp$ the Jordan triple product $\JTP{x}{y}{z}$ and the operators $D_{x,y}$ and $Q_{x,z}$ are defined by polarization of $Q_x$, 
\[
	\JTP{x}{y}{z}:=D_{x,y}z:=Q_{x,z}y:=Q_{x+z}y-Q_xy-Q_zy.
\]
The results of this section are independent of the choice of $\kk$. We refer to \cite{Lo75, Lo77} for a detailed introduction to Jordan pairs. We briefly recall some basic notions necessary for our purposes.

An element $\b e=(e_+,e_-)$ in $(V^+,V^-)$ is \emph{idempotent}, if $e_+ = Q_{e_+}e_-$ and $e_- = Q_{e_-}e_+$. The corresponding \emph{Peirce decomposition} is given by
\[
	V^\pm=V_2^\pm(\b e)\oplus V_1^\pm(\b e)\oplus V_0^\pm(\b e)\quad\text{with}\quad
	V_\nu^\pm(\b e):=\Set{z\in V^\pm}{D_{e_\pm,e_\mp}z = \nu\,z}.
\]
For the following fix $\sigma\in\{+,-\}$. The Peirce 2-space $V^\sigma_2(\b e)$ coincides with the \emph{principal inner ideal} $[e_\sigma]:=Q_{e_\sigma}V^{-\sigma}$ generated by $e_\sigma$. Moreover, $[e_\sigma]$ forms a unital Jordan algebra with product $x\circ y:=\frac{1}{2}\JTP{x}{e_{-\sigma}}{y}$ and unit element $e_\sigma$, see also Remark~\ref{rmk:principalinnerideal}. By definition, the \emph{Jordan algebra determinant} $\Delta_{\b e}^\sigma$ of $[e_\sigma]$ is the exact denominator of the rational map $z\mapsto z^{-1}$, normalized to $\Delta_\b e^\sigma(e_\sigma) = 1$. As usual, we expand $\Delta_{\b e}^\sigma$ to a polynomial on all of $V^\sigma$ by firstly projecting onto $[e_\sigma]$ (along $V_1^\sigma(\b e)\oplus V_0^\sigma(\b e)$) and then evaluating $\Delta_{\b e}^\sigma$. By abuse of notation, this expansion is also denoted by $\Delta_{\b e}^\sigma$ and called the \emph{generalized minor} associated to $\b e$.

There is also a determinant attached to the Jordan pair $(V^+,V^-)$, which we describe next. A pair $(x,y)\in V^\sigma\times V^{-\sigma}$ is \emph{quasi-invertible}, if the \emph{Bergman operator} $\B{x}{y}:=\Id - D_{x,y}+Q_xQ_y$ is invertible. In this case,
\[
	x^y:=\B{x}{y}^{-1}(x-Q_xy)
\]
is the \emph{quasi-inverse} of $(x,y)$. The exact denominator $\Delta:V^+\times V^-\to\kk$ of the rational map $(x,y)\mapsto x^y$, normalized to $\Delta(0,0)=1$, is the \emph{Jordan pair determinant} (also often called the \emph{generic norm}, see \cite[\S\,16.9]{Lo75}).

There is a simple relation between Jordan algebra determinants and the Jordan pair determinant, which we already noted in \cite{SS12} for the special case of complex simple Jordan pairs. To the best of our knowledge this relation has not been stated elsewhere.

\begin{proposition}\label{prop:MinorVsDeterminant}
	Let $\b e=(e^+,e^-)$ be an idempotent of $(V^+,V^-)$. Then
	\begin{align}\label{eq:MinorVsDeterminant}
		\Delta_\b e^+(x) = \Delta(e_+ - x,e_-),\quad
		\Delta_\b e^-(y) = \Delta(e_+,e_--y)
	\end{align}
	for all $x\in V^+$, $y\in V^-$.
\end{proposition}
\begin{proof}
Due to the duality of the Jordan pairs $(V^+,V^-)$ and $(V^-,V^+)$, it suffices to prove the first formula in \eqref{eq:MinorVsDeterminant}. If the Jordan pair is the one associated to a unital Jordan algebra $J$, i.e., $(V^+,V^-)=(J,J)$, and if $e_+$ is the unit element of $J$, then \eqref{eq:MinorVsDeterminant} is an immediate consequence of \cite[\S\,16.3(ii)]{Lo75}. We reduce the general case to this Jordan algebra case. Let $x=x_2+x_1+x_0$ be the decomposition of $x$ according to the Peirce decomposition with respect to $\b e$. By definition, $\Delta_\b e^+(x)=\Delta_\b e^+(x_2)$. On the other hand, due to \cite[\S\,3.5]{Lo75} the relation $\Delta(u,Q_vw) = \Delta(w,Q_vu)$ holds for all $u,w\in V^+$, $v\in V^-$, so it follows that $\Delta(e_+-x,e_-) = \Delta(e_+-x_2,e_-)$ since $e_-=Q_{e_-}e_+$ and $Q_{e_-}x=Q_{e_-}x_2$. Therefore it suffices to assume $x=x_2$. We may consider $([e_+],[e_-])$ as a subpair of $(V^+,V^-)$. Then its Jordan pair determinant coincides with the restriction of $\Delta$ to $[e_+]\times [e_-]$. Moreover, due to \cite[\S\,1.11]{Lo75} we may identify $([e_+],[e_-])$ with the Jordan pair $(J,J)$ with $J=[e_+]$ via the Jordan algebra isomorphism $Q_{e_-}:[e_+]\to [e_-]$. Now we are in the Jordan algebra case, and \eqref{eq:MinorVsDeterminant} follows from \cite[\S\,16.3(ii)]{Lo75}.
\end{proof}

\begin{remark}\label{rmk:principalinnerideal}
	We note that the Jordan algebra structure on $[e_\sigma]= V^\sigma_2(\b e)$ is independent of 
	$e_{-\sigma}$, since for $x=Q_{e_\sigma}u$ in $[e_\sigma]$, the fundamental formula for the 
	quadratic map yields
	\[
		x^2 = Q_xe_{-\sigma} = Q_{Q_{e_\sigma}u}e_{-\sigma} = Q_{e_\sigma}Q_uQ_{e_\sigma}e_{-\sigma}
		=Q_{e_\sigma}Q_ue_\sigma,
	\]
	and polarization of $x^2$ also shows that $x\circ y$ is independent of $e_{-\sigma}$. It follows
	that the Jordan algebra determinant $\Delta_{\b e}^\sigma:[e_\sigma]\to\kk$ does not depend on 
	$e_{-\sigma}$. However, its expansion to $V^\sigma$ depends on $e_{-\sigma}$ since the Peirce 
	spaces $V_1^\sigma(\b e)$ and $V_0^\sigma(\b e)$ are dependent on $e_{-\sigma}$. 
\end{remark}

\begin{example}\label{ex:usualminors}
	Consider the simple Jordan pair $(\kk^{r\times s},\kk^{s\times r})$ with $r\leq s$ and quadratic 
	maps given by $Q_xy=xyx$. Then, idempotents are pairs of matrices $(e_+,e_-)$ satisfying
	$e_+e_-e_+=e_-$ and $e_-e_+e_-=e_+$, and the Jordan pair determinant is given by $\Delta(x,y) = 
	\Det(\b 1_r - xy)$. For $1\leq k\leq r$ and tuples $I=(i_1,\ldots, i_k)$, $J=(j_1,\ldots,j_k)$ 
	with $1\leq i_1<\cdots<i_k\leq r$ and $1\leq j_1<\cdots<j_k\leq s$ let $e_+$ be the
	$r\times s$-matrix defined by 
	\begin{align}\label{eq:minortripotent}
		(e_+)_{ij} := \begin{cases}
			1 & \text{ if $(i,j) = (i_\ell,j_\ell)$ for some $1\leq\ell\leq k$,}\\
			0 & \text{ else,}
		\end{cases}
	\end{align}
	and set $e_-:=e_+^\top$, the transpose matrix of $e_+$. Then, it is straightforward to show that 
	$(e_+,e_-)$ is an idempotent, and the generalized minor 
	\[
		\Delta_\b e^+(z) = \Det(\b 1_r - (e_+-z)e_-)\quad(z\in\kk^{r\times s})
	\]
	coincides with the usual minor corresponding to rows and columns given by $I$ and $J$, 
	respectively.
\end{example}

\begin{proposition}\label{prop:JDetRelations}
	Let $\b e,\b c\in(V^+,V^-)$ be idempotents. 
	If $[e_+] = [c_+]$, then
	\begin{enumerate}[\upshape(i)]
		\item $\Delta_\b e^+(x) = \Delta_\b e^+(c_+)\cdot\Delta_\b c^+(x)$
					for all $x\in[e_+]$,\vspace{2mm}
		\item $\Delta_\b e^-(y)
					= \Delta_\b e^-(c_-)\cdot\Delta_\b c^-(y)$
					for all $y\in V^-$,\vspace{2mm}
		\item $\Delta_\b c^+(e_+)\cdot\Delta_\b c^-(e_-) = 1$.\vspace{2mm}
	\end{enumerate}
	If $[e_-]=[c_-]$, the same formulas hold when $+$ and $-$ are interchanged.
\end{proposition}

\begin{proof}
The first formula is well-known from the theory of mutations of Jordan algebras, see e.g.\ \cite[V.\S\,3]{BK66}. The second formula needs different arguments, since $[e_-]$ might differ from $[c_-]$. We claim that $(e_+,e_--c_-)$ is quasi-invertible with quasi-inverse $e_+^{e_--c_-}=c_+$. In this case, (ii) follows from standard identities of the Jordan pair determinant \cite[\S\,16.11]{Lo75},
\begin{align*}
	\Delta_\b e^-(y) &= \Delta(e_+,e_--y) \\
		&= \Delta(e_+,e_--c_-+c_--y) \\
		&= \Delta(e_+,e_--c_-)\Delta(e_+^{e_--c_-},c_--y) \\
		&= \Delta_\b e^-(c_-)\,\Delta(c_+,c_--y) \\
		&= \Delta_\b e^-(c_-)\,\Delta_\b c^-(y)\;.
\end{align*}
In order to show quasi-invertibility of $(e_+,e_--c_-)$, consider the decomposition $c_- = c_2\oplus c_1\oplus c_0$ of $c_-$ according to the Peirce decomposition of $V^-$ with respect to $\b e$. Since $\b c$ is an idempotent, the Peirce rules \cite[\S\,5.4]{Lo75} yield the following relations:
\[
	Q_{c_+}c_2 = c_+\;,\quad
	Q_{c_2}c_+ \oplus \JTP{c_2}{c_+}{c_1} \oplus Q_{c_1}c_+ = c_2\oplus c_1\oplus c_0\;.
\]
Comparing the components of the Peirce spaces in the second identity, we conclude that the pair $(c_+,c_2)$ is also idempotent with $[c_+]=[e_+]$ and $[c_2]=[e_-]$. By assumption, $c_+$ is invertible in the Jordan algebra $[e_+]$. Therefore, $c_2$ is invertible in the Jordan algebra $[e_-]$, and it follows that $(e_--c_2,e_+)$ is quasi-invertible with quasi-inverses
\[
	(e_--c_2)^{e_+} = c_2^{-1}-e_-,
\]
where $c_2^{-1}$ is the inverse of $c_2$ in $[e_-]$, see \cite[\S\,3.1]{Lo75}. Now recall that Jordan algebra inverses satisfy $a^{-1} = P_a^{-1}a$, where $P_a$ is the quadratic operator corresponding to $a\in[e_-]$. Here, $P_a = Q_aQ_{e_+}|_{[e_-]}$, so we obtain
\[
	c_2^{-1}=(Q_{c_2}Q_{e_+}|_{[e_-]})^{-1}c_2 = Q_{e_-}(Q_{c_2}|_{[c_2]})^{-1}c_2 = Q_{e_-}c_+.
\]
Due to the symmetry formula \cite[\S\,3.3]{Lo75} for quasi-inverses it follows that
\[
	e_+^{e_--c_2} = e_+ + Q_{e_+}(e_--c_2)^{e_+} = e_+ + Q_{e_+}(Q_{e_-}c_+ - e_-) = c_+.
\]
Finally, the shifting formula \cite[\S\,3.5]{Lo75} yields that $(e_+,e_--c_-)$ is quasi-invertible if and only if $(e_+,e_--c_2)$ is quasi-invertible, and both have the same quasi-inverse. It remains to show (iii). Since $\Delta(Q_uv,w) = \Delta(Q_uw,v)$ for all $u\in V^+$, $v,w\in V^-$, Proposition~\ref{prop:MinorVsDeterminant} yields the relation $\Delta_\b e^-(c_-) = \Delta_\b e^+(Q_{e_+}c_-)$. Now applying (i), we obtain
\begin{align}\label{eq:DetRelation}
	\Delta_\b e^-(c_-) = \Delta_{\b e}^+(c_+)\cdot\Delta_\b c^+(Q_{e_+}c_-).
\end{align}
Since $Q_{e_+}c_- = e_+^2$ in the Jordan algebra $[c_+]$, the second term on the right hand side becomes $\Delta_\b c^+(e_+)^2$. Moreover, setting $x=e_+$ in (i) yields $\Delta_\b c^+(e_+)=\Delta_\b e^+(c_+)^{-1}$, so \eqref{eq:DetRelation} implies (iii). By duality of the Jordan pairs $(V^+,V^-)$ and $(V^-,V^+)$ the same statement holds when $+$ and $-$ are interchanged.
\end{proof}

\begin{remark}
	We note that Proposition~\ref{prop:JDetRelations}(i) does not necessarily hold for all
	$x\in V^\sigma$, as the following calculation illustrates. As in Example~\ref{ex:usualminors}
	consider the Jordan pair $(\kk^{r\times s},\kk^{r\times s})$. Let $\b e = (e_+,e_-)$ be defined 
	by 
	\[
		e_+=\begin{pmatrix} A & 0 \\ 0 & 0\end{pmatrix},\quad
		e_-=\begin{pmatrix} A^{-1} & B \\ C & CAB\end{pmatrix},
	\]
	where $A\in\kk^{k\times k}$ is invertible, and $B\in\kk^{k\times(s-k)}$,
	$C\in\kk^{(r-k)\times k}$ are arbitrary. Any such pair of matrices is an idempotent of 
	$(\kk^{r\times s},\kk^{r\times s})$, and yields the same principal inner ideal in
	$\kk^{r\times s}$,
	\[
		[e_+] = \Set{\begin{pmatrix} a & 0\\0 & 0\end{pmatrix}}{a\in\kk^{k\times k}}.
	\]
	The generalized minors are given by
	\begin{align*}
		\Delta_\b e^+(x) &= \Det(\b 1_r - (e_+-x)e_-) 
			= \Det\begin{pmatrix} aA^{-1}+bC & -AB+aB+bCAB\\cA^{-1}+dC & 1+cB+dCAB\end{pmatrix}
		\intertext{where $x = \sMatrix{a}{b}{c}{d}\in\kk^{r\times s}$ with 
		$a\in\kk^{k\times k}$, and}
		\Delta_\b e^-(y) &= \Det(\b 1_r - e_+(e_--y))
			= \Det\begin{pmatrix} A\alpha & -A(B-\beta) \\0 & \b 1_{r-k}\end{pmatrix}
			= \Det A\cdot\Det\alpha
	\end{align*}
	where $y = \sMatrix{\alpha}{\beta}{\gamma}{\delta}\in\kk^{s\times r}$ with 
	$\alpha\in\kk^{k\times k}$. We see that $\Delta^+_\b e(x)$ simplifies to
	$\Det A^{-1}\cdot\Det a$ only if $x\in[e_+]$ or $B=0$ and $C=0$, so 
	Proposition~\ref{prop:JDetRelations}(i) fails for $x\neq[e_+]$, e.g.\ in the case where
	$B\neq 0$ and $\b c=\left(\sMatrix{1}{0}{0}{0}),\sMatrix{1}{0}{0}{0}\right)$.
\end{remark}

\begin{remark}
	Strictly speaking, Proposition~\ref{prop:JDetRelations}(ii) is not needed for the purpose of 
	this paper. However we find it worthwhile to include this formula, since it is a rather 
	suprising identity, in particular considered in contrast to 
	Proposition~\ref{prop:JDetRelations}(i) with its restricted domain.
\end{remark}

\section{Inequalities for generalized minors}\label{sec:inequalities}
In this section, let $(V,\Bar V)$ be a simple Jordan pair over $\kk=\RR$ or $\kk=\CC$ with positive involution $\vartheta:V\to \Bar V$. For convenience, we set $\bar z:=\vartheta(z)$ for $z\in V$. 
Recall that an element $e\in V$ is a \emph{tripotent}, if $\b e:=(e,\bar e)$ is an idempotent. All notions and results of the last section also apply to this idempotent, and since $e$ uniquely determines the idempotent, we simplify the notation and set
\[
	V_\nu(e):= V_\nu^+(\b e),\quad 
	\Delta_e(x):=\Delta_\b e^+(x),
\]
for $\nu=2,1,0$ and $x\in V$.

Two tripotents $e,c\in V$ are \emph{strongly orthogonal}, if $e\in V_0(c)$ (or equivalently $c\in V_0(e)$), and $e$ is called \emph{primitive}, if it cannot be written as a sum of two non-zero strongly orthogonal tripotents. Any tripotent is the sum of primitive tripotents, and the number of summands is called its \emph{rank}. A \emph{frame} is a maximal system $(e_1,\ldots, e_r)$ of primitive strongly orthogonal tripotents. Here, $r$ is independent of the choice of the frame, and called the \emph{rank} of the Jordan pair $V$.

Due to \cite[\S3.12]{Lo77} any element $z\in V$ admits a singular value decomposition, i.e.,
\begin{align}\label{eq:singvaldec}
	z = \sigma_1e_1+\cdots+\sigma_r e_r
\end{align}
where $(e_1,\ldots, e_r)$ is a frame of tripotents and $\sigma_1\geq\cdots\geq\sigma_r\geq 0$ are uniquely determined real numbers, called the \emph{singular values of $z$}.

\begin{example}
	In the case $V=\kk^{r\times s}$ discussed in our previous examples, a positive involution is 
	given by the Hermitian transpose, $z\mapsto z^*\in\kk^{s\times r}$. Then, an element 
	$e\in\kk^{r\times s}$ is a tripotent if and only if $e=ee^*e$, i.e., if $e$ is a partial 
	isometry. The element constructed in \eqref{eq:minortripotent} is in fact a tripotent. 
	The singular value decomposition \eqref{eq:singvaldec} coincides with the usual one. Indeed, let 
	$z=U_1\Sigma U_2$ be the usual singular value decomposition with $U_1\in\U_r(\kk)$, 
	$U_2\in\U_s(\kk)$ and diagonal matrix $\Sigma\in\kk^{r\times s}$ with entries 
	$\sigma_1\geq\cdots\geq\sigma_r\geq 0$. For	$1\leq k\leq r$ set $e_i:=U_1E_iU_2$ where $E_i$ 
	denotes the matrix with $1$ at the $(i,i)$'th position and $0$ elsewhere. Then
	$(e_1,\ldots, e_r)$ is a frame of tripotents, and $z=\sigma_1e_1+\cdots+\sigma_r e_r$.
\end{example}

\begin{theorem}\label{thm:MinorInequality}
	Let $e\in V$ be a tripotent of rank $k$, and $z\in V$ be an element with singular values 
	$\sigma_1\geq\cdots\geq\sigma_r\geq 0$. Then
	\[
		|\Delta_e(z)| \leq \sigma_1\cdots\sigma_k\,.
	\]
\end{theorem}

Before proving this, we determine the derivative of the Jordan pair determinant $\Delta$. Since $(V,\Bar V)$ is assumed to be simple, it follows from \cite[\S\,17.3]{Lo75} that $\Delta$ is irreducible and satisfies
\begin{align}\label{eq:JDetVsDetBerg}
	\Det\B{x}{y} = \Delta(x,y)^p,
\end{align}
where $p$ is a structure constant of $(V,\Bar V)$, and $\Det$ denotes the standard determinant of the Bergman operator $\B{x}{y}\in\End(V)$. Recall that 
\begin{align}\label{eq:traceform}
	\tau:V\times\Bar V\to\kk,\ (x,y)\mapsto \Tr D_{x,y}
\end{align}
is a non-degenerate pairing, called the \emph{trace form} of $(V,\Bar V)$.

\begin{lemma}\label{lem:JDderivative}
	The derivative of the Jordan pair determinant at $(x,y)\in V\times\Bar V$ along
	$(u,v)\in V\times\Bar V$ is given by
	\begin{align}\label{eq:JDderivative}
		d_{(u,v)}\Delta(x,y) = -\tfrac{1}{p}\,\Delta(x,y)\big(\tau(u,y^x)+\tau(x^y,v)\big)\,.
	\end{align}
\end{lemma}
\begin{proof}
We note that is suffices to prove \eqref{eq:JDderivative} for generic $(x,y)\in V\times\Bar V$, so we may assume that $\Delta(x,y)\neq 0$, or equivalently that $\B{x}{y}$ is invertible. Taking derivatives on both sides of \eqref{eq:JDetVsDetBerg} yields
\[
	p\,\Delta(x,y)^{p-1}\,d_{(u,v)}\Delta(x,y)
	= \Det\B{x}{y}\,\Tr\big(\B{x}{y}^{-1} d_{(u,v)}\B{x}{y}\big)\,,
\]
and hence
\[
	d_{(u,v)}\Delta(x,y) = \tfrac{1}{p}\Delta(x,y)\,\Tr\big(\B{x}{y}^{-1}d_{(u,v)}\B{x}{y}\big)\,.
\]
Since $d_{(u,v)}\B{x}{y} = -D_{u,y} + Q_{x,u}Q_y - D_{x,v}  + Q_xQ_{y,v}$, it follows that 
\[
	\Tr\big(\B{x}{y}^{-1} d_{(u,v)}\B{x}{y}\big)
	= -\Tr D_{u,y^x} - \Tr D_{x^y,v}
	= -\tau(u,y^x) -\tau(x^y,v)\,,
\]
where we used the relations $D_{u,y^x}\B{x}{y}=D_{u,y}-Q_{x,u}Q_y$ and $\B{x}{y}D_{x^y,v} = D_{x,v}-Q_xQ_{y,v}$, see the appendix of \cite{Lo77}. This completes the prove.
\end{proof}

\begin{proof}[Proof of Theorem~\ref{thm:MinorInequality}]
Let $S_k\subseteq V$ denote the subset of tripotents of rank $k$. It is known \cite[\S\S\,5.6, 11.12]{Lo77} that $S_k$ is a compact submanifold, and the tangent space $T_eS_k$ at $e\in S_k$ is given in terms of the Peirce decomposition of $V$ with respect to $e$ in the following way: Recall that the map $x\mapsto x^\#:=Q_e\bar x$ defines an involution on $V_2(e)$ with eigenspace decomposition $V_2(e) = A(e)\oplus B(e)$ where
\[
	A(e):=\Set{x\in V_2(e)}{x=x^\#},\quad
	B(e):=\Set{x\in V_2(e)}{x=-x^\#}.
\]
Then $T_eS_k = B(e)\oplus V_1(e)$. Moreover, recall that the decomposition $V=A(e)\oplus B(e)\oplus V_1(e)\oplus V_0(e)$ is orthogonal with respect to the positive definite inner product $(u,v):=\tau(u,\bar v)$ on $V$. For fixed $z\in V$, we determine the maximum value of the map $f:S_k\to\RR$, $f(e):=|\Delta_e(z)|^2$. Due to Lemma~\ref{prop:MinorVsDeterminant}, it is clear that $f$ is smooth, and Lemma~\ref{lem:JDderivative} implies that the derivative of $f$ at $e\in S_k$ along $u=u_2+u_1\in T_eS_k$ is given by
\[
	d_uf(e) = -\tfrac{2}{p}\,f(e)\cdot\Re\big(\tau(u,(\bar e)^{e-z}) + \tau((e-z)^{\bar e},\bar u)\big)\,.
\]
The symmetry formula for quasi-inverses \cite[\S\,3.3]{Lo75} yields $(\bar e)^{e-z} = \bar e + Q_{\bar e}(e-z)^{\bar e}$, and since $\tau(x,Q_yz) = \tau(z,Q_yx)$ it follows that
\[
	\tau(u,\bar e^{e-z}) = \tau(u,\bar e+Q_{\bar e}(e-z)^{\bar e})
	 = \tau(u,\bar e) + \tau((e-z)^{\bar e},Q_{\bar e} u).
\]
We thus obtain
\[
	d_uf(e) = -\tfrac{2}{p}\,f(e)\cdot\Re\tau((e-z)^{\bar e},\bar u_1)\,,
\]
since $e\in A(e)\perp T_eS_k$ and $Q_{\bar e}u + \bar u = \bar u_1$. Therefore, $f$ attains its maximum value at $e$ only if $\tau((e-z)^{\bar e},\bar u_1) = 0$ for all $u_1\in V_1(e)$, i.e., only if $(e-z)^{\bar e}=x_2+x_0\in V_2(e)\oplus V_0(e)$. In this case, the shifting formula for quasi-inverses \cite[\S\,3.5]{Lo75} yields
\[
	e-z = (x_2+x_0)^{-\bar e} = x_2^{-\bar e} + x_0\in V_2(e)\oplus V_0(e)\,,
\]
so it follows that $z=z_2+z_0\in V_2(e)\oplus V_0(e)$. Now, strong orthogonality of $V_2(e)$ and $V_0(e)$ implies that the singular value decomposition of $z$ splits into the corresponding decompositions of $z_2$ and $z_0$, so $\{\sigma_1,\ldots,\sigma_r\} = \{\mu_1,\ldots,\mu_k\}\cup\{\nu_1,\ldots,\nu_{r-k}\}$, where the $\mu_i$ (resp. $\nu_i$) are the singular values of $z_2$ (resp.\ $z_0$). We claim that $|\Delta_e(z)| = \prod\mu_i$. If so, then $f$ attains its maximum value if $\mu_i = \sigma_i$ for $i=1,\ldots, k$, and we are finished. To evaluate $|\Delta_e(z)|$ first recall that $\Delta_e(z) = \Delta_e(z_2)$. Let $z_2=\mu_1c_1+\cdots+\mu_kc_k$ denote the spectral decomposition of $z_2$. We note that this is not necessarily the spectral decomposition of $z_2\in[e]$ in the sense of Jordan algebras \cite[III.1.1]{FK90}, since the sum $c:=c_1+\cdots+c_k$ might differ from $e$. However, since $[c]=[e]$, Proposition~\ref{prop:JDetRelations}(i) yields $\Delta_e(z_2) = \Delta_c(e)^{-1}\Delta_c(z_2)$. Now, due to Proposition~\ref{prop:MinorVsDeterminant} and \cite[\S\,16.15]{Lo75} we obtain
\[
	\Delta_c(z_2) = \Delta(c-z_2,\bar c) = \prod(1-(1-\mu_i)\cdot 1) = \prod\mu_i.
\]
Finally, recall that $\Delta(u,\bar v) = \Bar{\Delta(v,\bar u)}$ for all $u,v\in V$, where $\Bar\alpha$ denotes complex conjugation of $\alpha\in\kk$. Therefore, Proposition~\ref{prop:JDetRelations}(iii) yields $|\Delta_c(e)|=1$, and we conclude that $|\Delta_e(z_2)| = \prod\mu_i$. This completes the proof.
\end{proof}

We finally give an application of Theorem~\ref{thm:MinorInequality} involving some representation theory. Consider a \emph{complex} simple Jordan pair $(V,\Bar V)$ with involution of rank $r$, and let $\Poly(V)$ denote the space of complex polynomial maps on $V$. Let $L$ denote the identity component of the structure group associated to $(V,\Bar V)$, which consists of linear maps $h\in\GL(V)$ satisfying the relation $h\JTP{x}{y}{z} = \JTP{hx}{h^{-\#}y}{hz}$ where $h^{-\#}$ denotes the inverse of the adjoint of $h$ with respect to the trace form $\tau$ defined in \eqref{eq:traceform}. Then $L$ is a reductive complex Lie group with maximal compact subgroup $K:=L\cap U(V)$, where $U(V)$ denotes unitary operators with respect to the inner product $(u,v)=\tau(u,\Bar v)$ on $V$. The induced action of $K$ on polynomials yields a decomposition of $\Poly(V)$ into irreducible components. Due to Hua, Kostant, Schmid \cite{Hu63,Sc69}, this decomposition is multiplicity free,
\[
	\Poly(V) = \bigoplus_{\b m\geq0}\Poly_\b m(V),
\]
and the irreducible components can be parametrized by tuples $\b m=(m_1,\ldots,m_r)$ of integers satisfying $m_1\geq\cdots\geq m_r\geq 0$ (corresponding to certain highest weights). As an application of Theorem~\ref{thm:MinorInequality}, we obtain a growth condition for polynomials in each component.

\begin{corollary}\label{cor:growthestimate}
	Let $z\in V$ be an element with singular values $\sigma_1\geq\cdots\geq\sigma_r\geq0$.
	For any $p\in\Poly_\b m(V)$ there exists $C>0$ such that
	\[
		|p(z)|\leq C\cdot\sigma_1^{m_1}\cdots\sigma_r^{m_r}.
	\]
\end{corollary}
\begin{proof}
Let $(e_1,\ldots,e_r)$ be a frame of tripotents. Recall from \cite{Up86} that the following polynomial map is a highest weight vector of $\Poly_\b m(V)$ (for an appropriate choice of a Borel subgroup of $L$),
\[
	p_\b m(z)
		:= \Delta_{\epsilon_1}(z)^{m_1-m_2}\cdot\Delta_{\epsilon_2}(z)^{m_2-m_3}\cdots
			 \Delta_{\epsilon_{r-1}}(z)^{m_{r-1}-m_r}\cdot
			 \Delta_{\epsilon_r}(z)^{m_r}
\]
where $\epsilon_k:=e_1+\cdots+e_k$. For $p_\b m$, Theorem~\ref{thm:MinorInequality} immediately yields 
\[
	|p_\b m(z)|\leq \sigma_1^{m_1}\cdots\sigma_r^{m_r}.
\]
For general $p\in\Poly_\b m(V)$, irreducibility implies that there are $c_i\in\CC$ and $k_i\in K$, such that 
\[
	p(z) = \sum_{i=1}^s c_i\cdot p_\b m(k_iz)
\]
Since singular values are invariant under the action of $K$, this proves our statement with $C:=\sum_i |c_i|$.
\end{proof}

We refer to \cite[Theorem~2.2]{S12b} for an application of this growth condition. The main advantage of this estimate is the $K$-invariance of the right hand side.

\bibliographystyle{amsplain}
\bibliography{bibdb}

\providecommand{\bysame}{\leavevmode\hbox to3em{\hrulefill}\thinspace}
\providecommand{\MR}{\relax\ifhmode\unskip\space\fi MR }
\providecommand{\MRhref}[2]{%
  \href{http://www.ams.org/mathscinet-getitem?mr=#1}{#2}
}
\providecommand{\href}[2]{#2}
\begin{thebibliography}{10}

\bibitem{BK66}
H.~Braun and M.~Koecher, \emph{{Jordan-Algebren}}, Grundlehren der
  mathematischen Wissenschaften, vol. 128, Springer-Verlag, Berlin--New York,
  1966.

\bibitem{Br56}
N.G.~de Bruijn, \emph{{Inequalities concerning minors and eigenvalues}}, Nieuw
  Archief v. Wiskunde \textbf{IV} (1956), no.~3, 18--35.

\bibitem{FK90}
J.~Faraut and A.~Koranyi, \emph{{Function spaces and reproducing kernels on
  bounded symmetric domains}}, J. Funct. Anal. \textbf{88} (1990), 64--89.

\bibitem{Hu63}
L.K. Hua, \emph{{Harmonic analysis of functions of several complex variables in
  the classical domains}}, Translations of Mathematical Monographs, vol.~6,
  American Mathematical Society, Providence, R.I., 1963.

\bibitem{Lo75}
O.~Loos, \emph{{Jordan Pairs}}, Lecture notes in Mathematics, vol. 460,
  Springer-Verlag, Berlin-New York, 1975.

\bibitem{Lo77}
\bysame, \emph{{Bounded symmetric domains and Jordan pairs}}, Lecture notes,
  University of California, Irvine, 1977.

\bibitem{Sc69}
W.~Schmid, \emph{{Die Randwerte holomorpher Funktionen auf hermitesch
  symmetrischen Räumen}}, Inv. math. \textbf{9} (1969), 61--80.

\bibitem{S12b}
B.~Schwarz, \emph{{Existence of nearly holomorphic sections on compact
  Hermitian symmetric spaces}}, preprint, available at
  \href{http://arxiv.org/}{arXiv.org}.

\bibitem{SS12}
B.~Schwarz and H.~Sepp\"anen, \emph{{Symplectic branching laws and Hermitian
  symmetric spaces}}, to appear in Trans. Amer. Math. Soc.

\bibitem{Up86}
H.~Upmeier, \emph{{Jordan algebras and harmonic analysis on symmetric spaces}},
  Amer. J. Math. \textbf{108} (1986), no.~1, 1--25.

\end{thebibliography}

\end{document}